\documentclass[11pt,a4paper,reqno]{amsart}
\usepackage{amsmath,amssymb, amsbsy}
\usepackage{subfigure}
\usepackage{graphpap,latexsym,epsf}
\usepackage{color,psfrag}
\usepackage[dvips]{graphicx}
\usepackage{enumerate}
\usepackage{bbm}
\usepackage{relsize}
\begin{document}
\newcommand{\dyle}{\displaystyle}
\newcommand{\R}{{\mathbb{R}}}
\newcommand{\Hi}{{\mathbb H}}
\newcommand{\Ss}{{\mathbb S}}
\newcommand{\N}{{\mathbb N}}
\newcommand{\Rn}{{\mathbb{R}^n}}
\newcommand{\F}{{\mathcal F}}
\newcommand{\ieq}{\begin{equation}}
\newcommand{\eeq}{\end{equation}}
\newcommand{\ieqa}{\begin{eqnarray}}
\newcommand{\eeqa}{\end{eqnarray}}
\newcommand{\ieqas}{\begin{eqnarray*}}
\newcommand{\eeqas}{\end{eqnarray*}}
\newcommand{\f}{\hat{f}}
\newcommand{\Bo}{\put(260,0){\rule{2mm}{2mm}}\\}
\newcommand{\1}{\mathlarger{\mathlarger{\mathbbm{1}}}}


\theoremstyle{plain}
\newtheorem{theorem}{Theorem} [section]
\newtheorem{corollary}[theorem]{Corollary}
\newtheorem{lemma}[theorem]{Lemma}
\newtheorem{proposition}[theorem]{Proposition}
\def\neweq#1{\begin{equation}\label{#1}}
\def\endeq{\end{equation}}
\def\eq#1{(\ref{#1})}


\theoremstyle{definition}
\newtheorem{definition}[theorem]{Definition}
\newtheorem{remark}[theorem]{Remark}
\numberwithin{figure}{section}

\title[It\^o vs Stratonovich]{It\^o vs Stratonovich in the presence of absorbing states}

\author[\'A. Correales, C. Escudero]{\'Alvaro Correales \& Carlos Escudero}
\address{}
\email{}


\date{\today}

\begin{abstract}
It is widely assumed that there exists a simple transformation from the It\^o interpretation to the one by Stratonovich and back for any
stochastic differential equation of applied interest. While this transformation exists under suitable conditions, and transforms one
interpretation into the other at the price of modifying the drift of the equation, it cannot be considered universal.
We show that a class of stochastic differential equations, characterized by the presence of absorbing states and of interest in applications,
does not admit such a transformation.
In particular, formally applying this transformation may lead to the disappearance of some absorbing states.
In turn, this modifies the long-time, and even the intermediate-time, behavior of the solutions.
The number of solutions can also be modified by the unjustified application of the mentioned transformation, as well as by
a change in the interpretation of the noise. We discuss how these facts affect the classical debate on the It\^o vs Stratonovich dilemma.
\end{abstract}
\maketitle

\section{Introduction}

Stochastic differential equations (SDEs) can be na\"{\i}vely regarded as nonautonomous dynamical systems of the sort
\begin{equation}\label{white}
dX_t = \mu(X_t,t) \, dt + \sigma(X_t,t) \, \xi_t,
\end{equation}
where the nonautonomous forcing $\xi_t$ is a random function known as \emph{white noise}: this stochastic process is supposed to take independent
and Gaussian distributed values at each time step. While this introduction of SDEs can be thought of as simple and useful for modeling, it lacks
any mathematical precision. It is the goal of classical stochastic analysis to give a precise meaning to such a model, and this goal can be
achieved by means of the introduction of a suitable stochastic integral. Among the infinitely many possibilities in this respect, two stand on a privileged
position, at least for historical reasons. These are the It\^o integral, that takes~\eqref{white} into
\begin{equation}\label{whiteito}
dX_t = \mu(X_t,t) \, dt + \sigma(X_t,t) \, dW_t,
\end{equation}
and the Stratonovich integral, which yields
\begin{equation}\label{whitestr}
dX_t = \mu(X_t,t) \, dt + \sigma(X_t,t) \circ dW_t;
\end{equation}
the properties of both have been studied profoundly~\cite{kuo,oksendal}. Obviously, it is a fundamental modeling question to be able to select the right
\emph{interpretation of noise} in a given application; in other words, it is key to distinguish which SDE, \eqref{whiteito} or~\eqref{whitestr},
is the right mathematical description of a particular natural or social phenomenon. In order to achieve this goal, it is necessary to understand
the analytical properties of both equations.

There is a widespread belief in the existence of a sort of \emph{analytical equivalence} between the It\^o and Stratonovich formulations of a SDE in the sense that a simple formula is able to connect both. In particular, the It\^o SDE
\begin{equation}\label{its}
d X_t = \mu(X_t,t) \, dt + \sigma(X_t,t) \, dW_t,
\end{equation}
is supposed to be equivalent to the Stratonovich SDE
\begin{equation}\label{sfi}
d X_t = \left[ \mu(X_t) - \frac12 \sigma'(X_t) \sigma(X_t) \right] dt + \sigma(X_t) \circ dW_t,
\end{equation}
with an analogous formula for the inverse transformation, under very general circumstances~\cite{mmcc}.
Indeed, for regular enough drift and diffusion terms this is
the case, but in order to highlight the importance of the counterexamples, let us state this fact in a precise manner.

\begin{theorem}\label{itotost}
Let $\mu(\cdot,\cdot): \mathbb{R} \times \mathbb{R}_+ \longrightarrow \mathbb{R}$
and $\sigma(\cdot,\cdot): \mathbb{R} \times \mathbb{R}_+ \longrightarrow \mathbb{R}$
be two globally Lipschitz continuous functions such that
$\sigma(\cdot,\cdot)$ is continuously differentiable with respect to the first variable and
$\sigma'(\cdot,\cdot) \sigma(\cdot,\cdot)$ is also globally Lipschitz continuous,
where the prime denotes differentiation with respect to the first variable.
Then the It\^o SDE
$$
dX_t = \mu(X_t,t) \, dt + \sigma(X_t,t) \, dW_t, \qquad X_0=x_0 \in \mathbb{R},
$$
and the Stratronovich SDE
$$
dY_t = \left[ \mu(Y_t,t) -\frac12 \sigma'(Y_t,t) \sigma(Y_t,t) \right] dt + \sigma(Y_t,t) \circ dW_t, \qquad Y_0=x_0 \in \mathbb{R},
$$
both possess a unique solution such that $X_t=Y_t$ for all $t \in [0,T]$ almost surely, for any $T>0$.
Correspondingly the Stratonovich SDE
$$
dX_t = \mu(X_t,t) \, dt + \sigma(X_t,t) \circ dW_t, \qquad X_0=x_0 \in \mathbb{R},
$$
and the It\^o SDE
$$
dY_t = \left[ \mu(Y_t,t) +\frac12 \sigma'(Y_t,t) \sigma(Y_t,t) \right] dt + \sigma(Y_t,t) \, dW_t, \qquad Y_0=x_0 \in \mathbb{R},
$$
both possess a unique solution such that $X_t=Y_t$ for all $t \in [0,T]$ almost surely, for any $T>0$.
\end{theorem}

This statement and its corresponding proof can be essentially found in~\cite{kuo}. We will see that the requirement of continuous differentiability of
$\sigma(\cdot,\cdot)$ cannot be weakened to the \emph{reasonable} case of $\sigma(\cdot,\cdot)$ being smooth almost everywhere even if
$\sigma'(\cdot,\cdot) \sigma(\cdot,\cdot)$ is both smooth and bounded. More importantly, we will not illustrate this fact with
pathological pure mathematical counterexamples, but with models that possess an applied interest.

The ultimate goal of the present work is to highlight a fact that seems to have been forgotten in the physical literature. Indeed, in reference~\cite{mmcc}
one reads: ``It is evident that stochasticians of all kinds - mathematicians, physicists, engineers
and others - need constant reminders that the It\^o versus Stratonovich problem
was solved long ago'', and the authors refer to reference~\cite{kampen}. Actually, one finds that in both references the full discussion regarding
the \emph{It\^o vs Stratonovich dilemma} is restricted to the transformation presented in Theorem~\ref{itotost}. Herein we will argue why any such
discussion is necessarily incomplete and how a change of interpretation or the formal application of the mentioned transformation may lead to the
appearance or disappearance of solutions. More importantly, we will try to convince the reader that our arguments, as well as the
\emph{It\^o interpretation of noise}, are not just some of ``those vagaries of the mathematical mind that are of no concern to him''~\cite{kampen}.

\subsection{Preliminaries and outline of the results}

Let us start making precise the theoretical framework we will use along this work.
First consider a Brownian motion $\{B_t, t \geq 0\}$ and a filtration $\{\F_t, t \geq 0\}$ such that:
\begin{itemize}
\item[(i)] For all $t \geq 0$, $B_t$ is $\F_t$-measurable,
\item[(ii)] for all $0 \leq s < t$, $B_t-B_s$ is independent of $\F_s$.
\end{itemize}
For the definition of the It\^o integral we take the following one.

\begin{definition}
Let $\{f(t), t\in [0,T]\}$ be a $\{\F_t\}$-adapted stochastic process. We define the \textit{It\^o stochastic integral} of $f(t)$ by
\begin{equation}\nonumber
\int_0^T f(t)\, dB_t : = \lim_{\left| \Pi_n \right| \to 0} \sum_{i=1}^n f(t_{i-1}) \left[B(t_i)-B(t_{i-1}) \right]
\end{equation}
in probability, provided that the limit exists, where the $\Pi_n$'s are partitions of the interval $[0,T]$
and $\left| \Pi_n \right|$ denote their diameters.
\end{definition}

For different definitions and properties of this integral see~\cite{kuo,oksendal}. We also use the following definition of Stratonovich integral.

\begin{definition}
Let $\{f(t), t\in [0,T]\}$ be a $\{\F_t\}$-adapted stochastic process. We define the \textit{Stratonovich stochastic integral} of $f(t)$ by
\begin{equation}\nonumber
\int_0^T f(t) \circ dB_t : = \lim_{\left| \Pi_n \right| \to 0} \sum_{i=1}^n f\left(\frac{t_i + t_{i-1}}{2}\right) \left[B(t_i)-B(t_{i-1}) \right]
\end{equation}
in probability, provided that the limit exists, where the $\Pi_n$'s are partitions of the interval $[0,T]$
and $\left| \Pi_n \right|$ denote their diameters.
\end{definition}

For more on the Stratonovich integral, including its extension to anticipating integrands, see~\cite{nualart,stratonovich}.

The remainder of this work is as follows. In Section~\ref{essence} we show that the transformation between noise interpretations present
in Theorem~\ref{itotost} cannot be considered as universal. In Section~\ref{longtime} we illustrate how this fact affects the long-time dynamics
of a suitable class of SDEs. Moreover, we show in Section~\ref{inttime} how this discrepancy in the dynamical behavior of a SDE can happen already
at intermediate times, highlighting its importance in applications. In Section~\ref{feller} we show that either changing the noise interpretation of a given
SDE or formally applying the transformation in Theorem~\ref{itotost} may lead to a modification in the number of solutions. Finally,
in Section~\ref{conclusions} we draw our main conclusions.

\section{Essence of the problem and one general result}
\label{essence}

Consider the SDE
\begin{equation}
\label{fbd}
dX_t = \sqrt{2 X_t} \, dW_t, \qquad X_0=0,
\end{equation}
which clearly admits the $X_t=0$ solution. This equation arises as the continuum limit of critical Galton-Watson branching
processes~\cite{feller1,feller2} and it is known as Feller branching diffusion~\cite{remark}.
A spatially extended version of this model has been used to study the growth and motion of plankton populations~\cite{adler},
and its trivial solution, which represents the extinction of the biological population, has a full physical meaning and should be present.
Let us consider its formal Stratonovich form
\begin{equation}
\label{illstr}
dX_t = -\frac12 \, dt + \sqrt{2 X_t} \circ dW_t, \qquad X_0=0,
\end{equation}
which clearly does not admit $X_t=0$ as solution. Correspondingly the Stratonovich SDE
\begin{equation}
\label{str2}
dX_t = \sqrt{2 X_t} \circ dW_t, \qquad X_0=0,
\end{equation}
admits the trivial solution but its formal It\^o counterpart
\begin{equation}
\label{ito2}
dX_t = \frac12 \, dt + \sqrt{2 X_t} \, dW_t, \qquad X_0=0,
\end{equation}
does not. We can generalize this fact to the following result.

\begin{theorem}\label{main1}
Let $f(\cdot,\cdot): \mathbb{R} \times \mathbb{R}_+ \longrightarrow \mathbb{R}$
and $g(\cdot,\cdot): \mathbb{R} \times \mathbb{R}_+ \longrightarrow \mathbb{R}_+$
be two continuous functions such that
$g'(\cdot,\cdot)$ is also continuous, where the prime denotes differentiation
with respect to the first argument. Assume also that
$f(x_0,t)=g(x_0,t)=0$ for all $t \ge 0$ and $g'(x_0,0) \neq 0$ for some $x_0 \in \mathbb{R}$.
Then the It\^o SDE
$$
dX_t = f(X_t,t) \, dt + \sqrt{2 g(X_t,t)} \, dW_t, \qquad X_0=x_0,
$$
admits the trivial solution $X_0=x_0$, but its formal Stratonovich dual
$$
dX_t = \left[ f(X_t,t) -\frac12 g'(X_t,t) \right] dt + \sqrt{2 g(X_t,t)} \circ dW_t, \qquad X_0=x_0,
$$
does not admit such a solution in the interval $[0,T]$ for any $T > 0$.
Correspondingly the Stratonovich SDE
$$
dX_t = f(X_t,t) \, dt + \sqrt{2 g(X_t,t)} \circ dW_t, \qquad X_0=x_0,
$$
admits the trivial solution $X_0=x_0$, but its formal It\^o dual
$$
dX_t = \left[ f(X_t,t) + \frac12 g'(X_t,t) \right] dt + \sqrt{2 g(X_t,t)} \, dW_t, \qquad X_0=x_0,
$$
does not admit such a solution in the interval $[0,T]$ for any $T > 0$.
\end{theorem}

\begin{proof}
We will prove explicitly the first affirmation as the proof of the second follows identically. The SDE
$$
dX_t = f(X_t,t) \, dt + \sqrt{2 g(X_t,t)} \, dW_t, \qquad X_0=x_0,
$$
actually means
$$
X_t = x_0 + \int_0^t f(X_s,s) \, ds + \int_0^t \sqrt{2 g(X_s,s)} \, dW_s, \qquad t \in [0,T].
$$
Substituting $X_t=x_0$ we find
$$
x_0 = x_0 + \int_0^t f(x_0,s) \, ds + \int_0^t \sqrt{2 g(x_0,s)} \, dW_s = x_0, \qquad t \in [0,T],
$$
so it is obviously a solution. On the other hand the SDE
$$
dX_t = \left[ f(X_t,t) -\frac12 g'(X_t,t) \right] dt + \sqrt{2 g(X_t,t)} \circ dW_t, \qquad X_0=x_0,
$$
actually means
$$
X_t = x_0 + \int_0^t \left[ f(X_s,s) -\frac12 g'(X_s,s) \right] ds + \int_0^t \sqrt{2 g(X_s,s)} \circ dW_s, \qquad t \in [0,T],
$$
so substituting $X_t=x_0$ yields
$$
x_0 = x_0 -\frac12 \int_0^t g'(x_0,s) \, ds \Longleftrightarrow \int_0^t g'(x_0,s) \, ds = 0,
$$
for every $t \in [0,T]$. Note that this integral is well-defined by continuity of $g'$ and assume without loss of generality that $g'(x_0,0)>0$.
Again by continuity we know that $g'(x_0,t) >0$ for $t \in [0,\delta)$ for some $\delta >0$ sufficiently small; then
$$
\int_0^t g'(x_0,s) \, ds >0 \quad \text{for every} \, t \in (0,\delta),
$$
hence $x_0$ is not a solution in [0,T] for any $T >0$.
\end{proof}

Another example with two absorbing states is the SDE
\begin{equation}\label{2as}
dX_t = \left( X_t - X_t^2 \right) dt + \sqrt{2(X_t - X_t^2)} \, dW_t,
\end{equation}
which clearly admits the solutions $X_t=0$ and $X_t=1$ when subject initially to these values.
However its formal Stratonovich counterpart
\begin{equation}\label{st2as}
dX_t = \left(-\frac12 + 2 X_t - X_t^2 \right) dt + \sqrt{2(X_t - X_t^2)} \circ dW_t,
\end{equation}
does not possess any of these two solutions. Note that this example, as well as the one described before, are simple consequences of Theorem~\ref{main1}.
Note also that, for the second example, the same conclusion arises if the noise interpretations are interchanged, which is also a simple consequence
of this theorem.

\section{Long-time dynamics}
\label{longtime}

In the present section we offer a dynamical approach to the question at hand. In particular we will show that the matter under study arises through the
presence of degenerate boundaries. To this end we need to introduce the following definitions.

\begin{definition}
A boundary point in the interval of definition of a SDE is called accessible if the probability of reaching that point in a finite time is positive.
\end{definition}

\begin{definition}
An accessible boundary point $\left\lbrace a\right\rbrace$ in the interval of definition of a SDE
is called instantaneously reflecting if and only if the Lebesgue measure of the set $\left\lbrace t: X_t = a\right\rbrace$ is zero almost surely,
is called slowly reflecting if and only if the Lebesgue measure of the set $\left\lbrace t: X_t = a\right\rbrace$ is finite almost surely,
and is called reflecting if it is either instantaneously or slowly reflecting.
\end{definition}

Let us consider again SDE~\eqref{fbd} but now subject to a different initial condition $x_0 \in \mathbb{R}^+$: in this case
the point $\left\lbrace 0\right\rbrace$ is an accessible boundary.
Together with the fact that $X_t = 0$ is a solution to this SDE when initialized
at this boundary (i.e. it is an absorbing boundary), this implies that $\lim_{t \to \infty} X_t =0$ with probability one. Indeed,
the solution $X_t= Z_t/2$, where $Z_t$ denotes the $0-$dimensional squared Bessel process. Using the theory developed for this process we
conclude that $\lim_{t \to \infty} X_t =0$ a.s.~\cite{revuz}.
In sharp contrast and as we have already seen, its formal Stratonovich counterpart~\eqref{illstr} does not admit the trivial solution.
Correspondingly, for the SDE~\eqref{ito2} initialized at $x_0 \in \mathbb{R}^+$ the origin is accessible but instantaneously reflecting,
as in this case $X_t= \tilde{Z}_t/2$, where $\tilde{Z}_t$ denotes the $1-$dimensional squared Bessel process~\cite{revuz}.
On the other hand its formal Stratonovich counterpart~\eqref{str2} admits $X_t=0$ as solution.
This summarizes the problematic discussed in the previous section from a different viewpoint, but we still need to face the general situation.

From now on we focus on equations of the type
$$
dX_t = f(X_t) \, dt + \sqrt{2 g(X_t)} \, dW_t, \qquad X_0=\gamma_0 \in \mathbb{R},
$$
that is, on time-homogeneous diffusions,
but prior to analyzing their dynamical behavior, we need a result that precisely states their well-posedness.

\begin{definition}
Let $h: \left[a,\infty \right) \longrightarrow U \subseteq \mathbb{R}$, with $a \in \mathbb{R}$. We say that $h \in C^2\left[a,\infty \right)$
whenever there exist an $\epsilon >0$ and a function $h_{\epsilon} \in C^2\left(a-\epsilon,\infty\right)$ such that $h_{\epsilon}=h$
in $\left[a,\infty \right)$. Moreover we write $h \in BC^2\left[a,\infty \right)$ if the function $h$ along with its first and second derivatives are bounded in its domain of definition, with the understanding that $h'(a):=h_\epsilon'(a)$ and $h''(a):=h_\epsilon''(a)$.
\end{definition}

\begin{remark}
Although we will explicitly assume that the drift and diffusion terms of the SDEs under consideration belong to $BC^2\left[a,\infty \right)$,
this can be relaxed to assuming that their first and second derivatives are bounded, but the functions themselves just obey the linear growth
condition (i.e. they can be absolutely bounded by an affine function). All the results in this work still hold under this milder requirement,
as it already forbids finite time blow-ups, see for instance~\cite{jean}.
\end{remark}

\begin{proposition}\label{exun}
Let $f: \left[a,\infty \right) \longrightarrow \mathbb{R}$ and $g: \left[a,\infty \right)  \longrightarrow \mathbb{R}_+$,
with $a \in \mathbb{R}$, and $f,g \in BC^2\left[a,\infty \right)$. Assume that $g(x)=0$ if and only if $x=a$, with $g'(a)\neq 0$,
and that there exist constants $C,\delta >0$ such that
$g(x) \geq C$ for $x \geq a+\delta$. Additionally assume that the function $f$ satisfies the compatibility condition $f(a) \ge 0$.
Then, the It\^o SDE
\begin{equation}\label{theo_eq_0}
dX_t = f(X_t) \, dt + \sqrt{2 g(X_t)} \, dW_t, \qquad X_0 = x_0 \in [a,\infty),
\end{equation}
possesses a unique strong solution for all $t \geq 0$.	
\end{proposition}

\begin{proof}
First observe that $f,g \in BC^2\left[a,\infty \right)$ implies they are globally Lipschitz and satisfy the linear growth condition. This last condition
is also satisfied by $\sqrt{g}$, and moreover, if $x,y \geq a+\delta$ then
\begin{equation*}
\left\lvert \sqrt{g(x)}-\sqrt{g(y)}\right\rvert \leq \sup_{z\geq a+\delta}\left\lbrace \dfrac{g'(z)}{2 \sqrt{g(z)}}\right\rbrace \left\lvert x-y\right\rvert.
\end{equation*}
This inequality implies that $\sqrt{g}$ is globally Lipschitz for $x \geq a+\delta$. So the lapse of existence of any local in time solution,
if it exists, can be arbitrarily extended, see for instance~\cite{jean}.

Now, the string of inequalities
\begin{eqnarray*}
\left\lvert \sqrt{x}-\sqrt{y}\right\rvert^2 &=& x+y-2\sqrt{xy} \\ \nonumber
&\leq& x+y-2\min\left\lbrace x,y\right\rbrace \\ \nonumber
&\leq& \max\left\lbrace x,y\right\rbrace-\min\left\lbrace x,y\right\rbrace \\ \nonumber
&=& \left\lvert x-y\right\rvert
\end{eqnarray*}
implies
\begin{displaymath}
\left\vert \sqrt{g(x)}-\sqrt{g(y)} \right\vert \leq \sqrt{\left\vert g(x)-g(y)\right\vert}.
\end{displaymath}
These results, together with the classical theorem of Watanabe and Yamada~\cite{ander,boro,jean,sv,wy}, imply existence and uniqueness of a strong solution to the SDE~\eqref{theo_eq_0}, which is defined for all times $t>0$.
\end{proof}

\begin{remark}
We are implicitly assuming all over this work that $g \ge 0$ so all the It\^o diffusions we consider are real-valued. In particular,
this implies $g'(a)>0$.
\end{remark}

\begin{remark}\label{remm}
An analogous result holds {\it mutatis mutandis} if the domain of definition of $f$ and $g$ is shifted to either $\left( -\infty, b\right]$ or
$\left[ a,b\right]$. For instance, the compatibility condition for the right boundary point should be $f(b)\leq 0$, i.e. the reflecting properties of the boundary should be preserved. From now on (unless explicitly indicated) we will state our results just for the original case of domain
$\left[ a, \infty \right)$, but with the understanding that they still hold for any of the other two after simple modifications.
\end{remark}

The following theorem characterizes the boundary behavior of the solution to an It\^o SDE of type~\eqref{theo_eq_0} and compares it with the properties of its formal Stratonovich counterpart.

\begin{theorem}\label{main2}
Let $f$ and $g$ be functions satisfying the same assumptions as in Proposition~\ref{exun}.
The boundary $\left\lbrace a\right\rbrace$ is accessible for the It\^o SDE
\begin{equation}\label{theo_eq_1}
dX_t = f(X_t) \, dt + \sqrt{2 g(X_t)} \, dW_t
\end{equation}
if and only if $f(a)< g'(a)$, and absorbing if and only if $f(a)=0$.
On the other hand, its formal Stratonovich dual
\begin{equation}\label{theo_eq_2}
dX_t = \left[ f(X_t) -\frac12 g'(X_t) \right] dt + \sqrt{2 g(X_t)} \circ dW_t
\end{equation}
only admits the constant solution $X_t=a$ if $f(a)=\frac{1}{2}g'(a)< g'(a)$.
\end{theorem}

\begin{remark}
This theorem implies that SDE~\eqref{theo_eq_2} admits the constant solution $X_t=a$ only when SDE~\eqref{theo_eq_1} has a reflecting boundary at $a$;
respectively, when this boundary is absorbent, SDE~\eqref{theo_eq_2} does not admit this constant solution.
\end{remark}

\begin{remark}
From now on we will adopt the convention of denoting all constants by the letter $C$, independently of their specific value.
\end{remark}

\begin{proof}
The boundary behavior is studied by means of the method of speed/scale measures~\cite{ander,boro,hell}.
According to it, the boundary $\left\lbrace a \right\rbrace$ is accessible for equation~\eqref{theo_eq_1} if and only if
\begin{displaymath}
\int_{a}^{a+\delta} \int_x^{a+\delta}\dfrac{1}{g(y)} \, e^{\int_x^y\frac{f(s)}{g(s)} \, ds} \, dy \, dx<\infty,
\end{displaymath}
for some constant $\delta>0$. We consider the cases $f(a)=0$ (for which the inequality $f(a) < g'(a)$ is automatically guaranteed) and $f(a)>0$ separately.

First we assume that $f(a)=0$. Then by L'H\^opital rule we find
$$
\lim_{x\rightarrow a^+}\frac{f(x)}{g(x)}=\frac{f'(a)}{g'(a)} < \infty.
$$
So the integral $\int_x^y\frac{f(s)}{g(s)} \, ds$ is finite for any fixed $\delta$ since $y \leq a+\delta$. Moreover we have the estimate
\begin{eqnarray}\nonumber
\int_{a}^{a+\delta} \int_x^{a+\delta}\dfrac{1}{g(y)} \, e^{\int_x^y\frac{f(s)}{g(s)} \, ds} \, dy \, dx &<&
C\int_{a}^{a+\delta} \int_x^{a+\delta}\dfrac{dy \, dx}{g(y)} \\ \nonumber
&=& C\int_{a}^{a+\delta} \int_{a}^{y}\dfrac{dx \, dy}{g(y)} \\ \nonumber
&=& C\int_{a}^{a+\delta} \dfrac{y-a}{g(y)} \, dy < \infty,
\end{eqnarray}
where the first inequality comes from the boundedness of $\int_x^y\frac{f(s)}{g(s)} \, ds$ and the last one from the second application
of L'H\^opital rule $\lim_{y\rightarrow a^+}\frac{y-a}{g(y)}=\frac{1}{g'(a)}$. Thus $\left\lbrace a\right\rbrace$ is an accessible boundary for $f(a)=0$. Since $X_t=a$ is a solution to~\eqref{theo_eq_1} in this case, then by uniqueness of solution we conclude this boundary is absorbing.

Now we turn to examining the finiteness of the integral
\begin{displaymath}
\int_{a}^{a+\delta} \int_x^{a+\delta}\dfrac{1}{g(y)} \, e^{\int_x^y\frac{f(s)}{g(s)} \, ds} \, dy \, dx
\end{displaymath}
in the case $f(a)>0$; for a small enough $\epsilon > 0$ (clearly it should be sufficiently smaller than $\delta$) we find
\begin{eqnarray*}
\int_{a+\epsilon}^{a+\delta} \int_x^{a+\delta}\dfrac{1}{g(y)} \, e^{\int_x^y\frac{f(s)}{g(s)} \, ds} \, dy \, dx &=&
\int_{a+\epsilon}^{a+\delta} \int_x^{a+\delta}\dfrac{1}{f(y)} \, \dfrac{f(y)}{g(y)}\, e^{\int_x^y\frac{f(s)}{g(s)} \, ds} \, dy \, dx \\
&\geq& C\int_{a+\epsilon}^{a+\delta} \int_x^{a+\delta}\dfrac{f(y)}{g(y)} \, e^{\int_x^y\frac{f(s)}{g(s)} \, ds} \, dy \, dx \\
&=& C\int_{a+\epsilon}^{a+\delta} \int_x^{a+\delta}\dfrac{d}{dy} \, e^{\int_x^y\frac{f(s)}{g(s)} \, ds} \, dy \, dx \\
&=& C\int_{a+\epsilon}^{a+\delta} \left( e^{\int_x^{a+\delta}\frac{f(s)}{g(s)} \, ds}-1 \right)dx \\ &=& C\int_{a+\epsilon}^{a+\delta} \left( e^{\int_x^{a+\delta}\frac{(s-a)f(s)}{g(s)}\frac{ds}{s-a}} \right)dx-\delta \\ &=& C\int_{a+\epsilon}^{a+\delta} \left( e^{\int_x^{a+\delta}\frac{(s-a)f(a)}{g(s)}\frac{ds}{s-a}+ \int_x^{a+\delta} \frac{(s-a)[f(s)-f(a)]}{g(s)}\frac{ds}{s-a}} \right)dx-\delta \\
&\geq& C\int_{a+\epsilon}^{a+\delta} \left( e^{\int_x^{a+\delta}\frac{(s-a)f(a)}{g(s)}\frac{ds}{s-a}} \right)dx-\delta,
\end{eqnarray*}
where the first inequality comes from the sign of $f(a)$ by choosing $\delta$ small enough, and the second comes from the boundedness of
$\int_x^{a+\delta} \frac{f(s)-f(a)}{g(s)} \, ds$ in the range $x \in [a,a+\delta]$ (i.e. the bound is uniform in $\epsilon$)
as can be seen by means of the application of L'H\^opital rule $\lim_{s \rightarrow a^+}\frac{f(s)-f(a)}{g(s)}=\frac{f'(a)}{g'(a)}$.
Now use the Taylor expansion $g(s)=g'(a)(s-a)+\frac12 g''(\bar{a})(s-a)^2$, where $\bar{a} \in [a,a+\delta]$, and choose a $\delta$ small enough to find
\begin{eqnarray}\nonumber
\int_{a + \epsilon}^{a+\delta} \left( e^{\int_x^{a+\delta}\frac{(s-a)f(a)}{g(s)}\frac{ds}{s-a}} \right)dx-\delta &\geq&
C\int_{a + \epsilon}^{a+\delta} \left( e^{\int_x^{a+\delta}\frac{f(a)}{g'(a)}\frac{ds}{s-a}} \right)dx-\delta \\ \nonumber
&=& C\int_{a + \epsilon}^{a+\delta} \left( \dfrac{\delta}{x-a}\right)^{\frac{f(a)}{g'(a)}}dx-\delta;
\end{eqnarray}
so by taking the limit $\epsilon \to 0$ we establish
\begin{eqnarray}\nonumber
\int_{a}^{a+\delta} \int_x^{a+\delta}\dfrac{1}{g(y)} \, e^{\int_x^y\frac{f(s)}{g(s)} \, ds} \, dy \, dx &\geq&
C\int_{a}^{a+\delta} \left( \dfrac{\delta}{x-a}\right)^{\frac{f(a)}{g'(a)}}dx-\delta
\\ \nonumber
&=& \infty \Longleftrightarrow \, f(a)\geq g'(a).
\end{eqnarray}
Arguing analogously one also establishes
\begin{eqnarray}\nonumber
\int_{a}^{a+\delta} \int_x^{a+\delta}\dfrac{1}{g(y)} \, e^{\int_x^y\frac{f(s)}{g(s)} \, ds} \, dy \, dx &\leq&
C\int_{a}^{a+\delta} \left( \dfrac{\delta}{x-a}\right)^{\frac{f(a)}{g'(a)}}dx-\delta
\\ \nonumber
&<& \infty \Longleftrightarrow \, f(a) < g'(a).
\end{eqnarray}
Combining these two results we conclude that $\{a\}$ is accessible if and only if $f(a)<g'(a)$.
Moreover, together with the compatibility condition $f(a)>0$ and uniqueness of solution we find that the boundary is instantaneously reflecting
in this case \cite{ander2}.

Finally, the part of the statement that corresponds to the Stratonovich SDE is obvious.
\end{proof}

To further exemplify this theorem we reconsider equation~\eqref{2as}. For any initial condition $x_0 \in [0,1]$, according to Theorem~\ref{main2}
(see also Remark~\ref{remm}), the boundary points $\{0\}$ and $\{1\}$ are both accessible and absorbing. Therefore, by Markovianity, we conclude that $\lim_{t \to \infty} X_t \in \{0,1\}$ a.s. On the other hand, the formal Stratonovich counterpart of this equation, i.e. equation~\eqref{st2as}, does not possess either of these solutions.

\section{Intermediate-time dynamics}
\label{inttime}

In this section we show that the results concerning the long-time dynamics proven in the previous section may have also consequences at intermediate times.
In particular, we prove that the mean time to absorption is finite in some cases, what means that there exists a well-defined time scale to absorption.	

For the sake of simplicity let us start with the SDE
\begin{equation*}
dX_t = \sqrt{2 X_t(1-X_t)} \, dW_t, \qquad X_0=x_0,
\end{equation*}
which falls under the assumptions of Theorems~\ref{main1} and~\ref{main2}.
Denote by $T(x_0)$ the mean time to absorption; it obeys the partial differential equation
\begin{equation*}
x_0(1-x_0) \, \dfrac{\partial^2 T}{\partial x_0^2}=-1
\end{equation*}
subject to the boundary conditions $T(0)=T(1)=0$~\cite{hell}.
The explicit solution to this boundary value problem is easy computable and reads
\begin{equation*}
T(x_0)=-(1-x_0) \log \left( 1-x_0\right)-x_0 \log \left( x_0\right),
\end{equation*}
which is clearly bounded uniformly in $x_0 \in [0,1]$. Therefore absorption has a well-defined time scale and consequently the asymptotic behavior
mentioned in the previous section will be reached at intermediate times.

Now we can move back to our previous example~\eqref{2as}, i.e.
\begin{equation*}
dX_t = X_t(1-X_t) \, dt + \sqrt{2 X_t(1-X_t)} \, dW_t, \qquad X_0=x_0,
\end{equation*}
which again falls under the hypotheses of Theorems~\ref{main1} and~\ref{main2}.
Now the mean time to absorption obeys the partial differential equation
\begin{equation*}
x_0(1-x_0) \, \dfrac{\partial T}{\partial x_0} + x_0(1-x_0) \, \dfrac{\partial^2 T}{\partial x_0^2}=-1
\end{equation*}
subject to $T(0)=T(1)=0$. Its solution reads
\begin{equation*}
T(x_0)= \frac{e^{-x_0}-1}{e-1} \int_{0}^{1} e^{y} \, \log \left( \dfrac{1-y}{y}\right) \, dy
+ e^{ -x_0} \int_{0}^{x_0} e^{y} \, \log \left( \dfrac{1-y}{y}\right) \, dy.
\end{equation*}
The uniform boundedness of this result for $x_0 \in [0,1]$ is guaranteed by the summable character of the integrands.
So the existence of a well-defined time scale to absorption arises in this example too.

We can generalize these examples through the following statement.

\begin{theorem}\label{mat}
Let $f$ and $g$ be functions as in Proposition~\ref{exun}, Remark~\ref{remm}, and Theorem~\ref{main2}.
Then the solution to the It\^o SDE
\begin{equation}\nonumber
dX_t = f(X_t) \, dt + \sqrt{2 g(X_t)} \, dW_t, \qquad x_0 \in [a,b],
\end{equation}
subject to a finite state space $[a,b]$ in which both boundary points are accessible, and at least one is absorbing and the other is either absorbing or reflecting, is absorbed in finite mean time.
\end{theorem}

\begin{proof}
By Proposition~\ref{exun} we know there exists a unique and global strong solution $X_t$ to this equation.
Denote by $\mathcal{A}$ the set of absorbing boundary points that, according to our assumptions, could be
either $\{a\}$, $\{b\}$, or $\{a,b\}$. For every initial condition $\gamma_0 \in [a,b]$ define
$$
\epsilon_T(\gamma_0):= P_{\gamma_0}(X_T \in \mathcal{A}),
$$
so $\epsilon_T(\gamma_0)>0$ for every large enough $T>0$ by assumption.
Now assume $\{b\}$ is absorbing to find
$$
\epsilon_T(\gamma_0) \ge P_{\gamma_0}(X_T =b) \ge P_{x_0}(X_T =b)
$$
for any $x_0 \le \gamma_0$ by the $t-$continuity and Markovianity of $X_t$~\cite{jean,revuz}. Analogously if $\{a\}$ is absorbing then
$$
\epsilon_T(\gamma_0) \ge P_{\gamma_0}(X_T =a) \ge P_{x_0}(X_T =a)
$$
for any $x_0 \ge \gamma_0$. From now on we consider a $T$ large enough so $\min\{P_{x_0}(X_T =b),P_{x_0}(X_T =a)\}>0$ for every given
initial condition $x_0$, which is always possible by the assumption on the accessibility of both boundary points.

For the time being let us assume that $\mathcal{A} \equiv \{a,b\}$. Now fix some $c$, $a < c < b$, to find
\begin{eqnarray}\nonumber
\inf_{x_0 \in [a,b]} P_{x_0}(X_T \in \mathcal{A}) &=& \min \left\{ \inf_{x_0 \in [a,c]} P_{x_0}(X_T \in \mathcal{A}),
\inf_{x_0 \in [c,b]} P_{x_0}(X_T \in \mathcal{A}) \right\}
\\ \nonumber
&\ge& \min \left\{ \inf_{x_0 \in [a,c]} P_{x_0}(X_T =a),
\inf_{x_0 \in [c,b]} P_{x_0}(X_T = b) \right\}
\\ \nonumber
&\ge& \min \left\{ P_{c}(X_T =a), P_{c}(X_T = b) \right\} >0,
\end{eqnarray}
where we have employed, as in the previous paragraph, the Markovianity and $t-$continuity of $X_t$, and where a larger enough $T$ has been selected,
in case that had been necessary. Therefore we can define
$$
\varepsilon_T := \inf_{\gamma_0 \in [a,b]} \epsilon_T(\gamma_0)
$$
as a positive quantity. Note that the same conclusion arises analogously in the cases $\mathcal{A} \equiv \{a\}$ and $\mathcal{A} \equiv \{b\}$.

From now on we fix the value of $T$; then compute
\begin{eqnarray}\nonumber
P\left(X_{(n+1)T} \in \mathcal{A}\right) &=& P\left(X_{(n+1)T} \in \mathcal{A} \cap X_{nT} \in \mathcal{A}\right)
+ P\left(X_{(n+1)T} \in \mathcal{A} \cap X_{nT} \not\in \mathcal{A}\right) \\ \nonumber
&=& P\left(X_{(n+1)T} \in \mathcal{A} \Big| X_{nT} \in \mathcal{A}\right)P\left(X_{nT} \in \mathcal{A}\right)
\\ \nonumber & &
+ P\left(X_{(n+1)T} \in \mathcal{A} \Big| X_{nT} \not\in \mathcal{A}\right)
P\left(X_{nT} \not\in \mathcal{A}\right) \\ \nonumber
&=& P\left(X_{nT} \in \mathcal{A}\right)
+ P\left(X_{(n+1)T} \in \mathcal{A} \Big| X_{nT} \not\in \mathcal{A}\right)
\left[ 1-P\left(X_{nT} \in \mathcal{A}\right) \right] \\ \nonumber
&\ge& P\left(X_{nT} \in \mathcal{A}\right) + \varepsilon_T \left[ 1-P\left(X_{nT} \in \mathcal{A}\right) \right]
\end{eqnarray}
for $n=1,2,\cdots$,
where the inequality follows from the Markov property of $X_t$, and obviously $P\left(X_{T} \in \mathcal{A}\right) \ge \varepsilon_T$.
Compare this inequality with the recursion relation
\begin{eqnarray}\nonumber
U(n+1) &=& U(n) + \varepsilon_T [1-U(n)] \\ \nonumber
U(1) &=& \varepsilon_T
\end{eqnarray}
that can be solved to yield
$$
U(n)=1-(1-\varepsilon_T)^n
$$
for all $n \ge 1$. Clearly $P\left(X_{T} \in \mathcal{A}\right) \ge U(1)$ and by induction
\begin{eqnarray}\nonumber
P\left(X_{(n+1)T} \in \mathcal{A}\right) &\ge& P\left(X_{nT} \in \mathcal{A}\right) + \varepsilon_T \left[ 1-P\left(X_{nT} \in \mathcal{A}\right) \right]
\\ \nonumber
&=& \varepsilon_T + P\left(X_{nT} \in \mathcal{A}\right) \left[ 1-\varepsilon_T \right] \\ \nonumber
&\ge& \varepsilon_T + U(n) \left[ 1-\varepsilon_T \right] \\ \nonumber
&=& U(n) + \varepsilon_T \left[ 1-U(n) \right] \\ \nonumber
&=& U(n+1),
\end{eqnarray}
therefore
$$
P\left(X_{nT} \in \mathcal{A}\right) \ge 1-(1-\varepsilon_T)^n \longrightarrow 1
$$
when $n \to \infty$, so adsorption happens almost surely in the long time limit.

Now define
$$
\tau_{x_0}(\omega):= \inf \left\{ t \ge 0 \, : \, \left( X_{t}|X_{0}=x_0 \right) \in \mathcal{A} \right\},
$$
that is, the random variable $\tau_{x_0}(\omega)$ is the exit time from the set $\mathcal{A}^c$ or, in other words, the absorption time.
We compute
\begin{eqnarray}\nonumber
\mathbb{E} (\tau_{x_0}) &=& \int_0^\infty P\left( \tau_{x_0}>t \right) \, dt \\ \nonumber
&=& \sum_{n=0}^\infty \int_{nT}^{(n+1)T} P\left( \tau_{x_0}>t \right) \, dt \\ \nonumber
&\le& T \sum_{n=0}^\infty P\left( \tau_{x_0}> nT \right) \\ \nonumber
&=& T \sum_{n=0}^\infty P\left( X_{nT} \not\in \mathcal{A} \right) \\ \nonumber
&=& T \sum_{n=0}^\infty \left[ 1 - P\left( X_{nT} \in \mathcal{A} \right) \right] \\ \nonumber
&\le& T \sum_{n=0}^\infty (1-\varepsilon_T)^n \\ \nonumber
&=& \frac{T}{\varepsilon_T} < \infty
\end{eqnarray}
for each $x_0 \in [a,b]$, so the statement follows.
\end{proof}

\section{Back to the Feller branching diffusion}
\label{feller}

Let us reconsider now our first example
\begin{equation}\label{fbd2}
dX_t=\sqrt{2 X_t} \, dW_t, \qquad X_0=x_0 > 0.
\end{equation}
We will see that Theorem~\ref{mat} is sharp in the sense that it is not guaranteed that It\^o diffusions with unbounded state spaces possess finite
mean absorption times. This is the case of the Feller branching diffusion, which state space is $[0,\infty)$.
To compute the mean exit time of the Feller branching diffusion from the interval $\left( 0, M \right)$ one needs to solve the boundary value problem
\begin{equation*}
x_0 \, \dfrac{\partial^2 T_M}{\partial x_0^2}=-1, \qquad T_M(0)=T_M(M)=0;
\end{equation*}
its explicit solution reads
\begin{equation*}
T_M(x_0)= x_0 \, \log \left( \frac{M}{x_0} \right).
\end{equation*}
It becomes unbounded as $M \rightarrow \infty$ for any $x_0 >0$, meaning that the mean absorption time $T(x_0)$ diverges.

Additionally, the Feller branching diffusion constitutes an excellent example of how analytical issues arise in the It\^o vs Stratonovich dilemma. Of course,
the SDE
\begin{equation}\nonumber
dX_t=\sqrt{2 X_t} \, dW_t
\end{equation}
is doubtless an It\^o diffusion since, as mentioned before, it arises as the continuum limit of critical
Galton-Watson branching processes, so its interpretation is not questionable. However, as a theoretical exercise,
we can consider the Stratonovich SDE
\begin{equation}\label{stfbd}
dX_t=\sqrt{2 X_t} \circ dW_t, \qquad X_0=x_0 > 0;
\end{equation}
which means
\begin{equation}\nonumber
X_t = x_0 + \int_0^t \sqrt{2 X_s} \circ dW_s.
\end{equation}
This equation clearly admits the solution
\begin{equation}\label{solun}
X_t= \left( \sqrt{x_0} + \frac{W_t}{\sqrt{2}} \right)^2.
\end{equation}
It is obvious that this $X_t$ will reach zero almost surely (although in infinite mean time), and also that zero is an instantaneously
reflecting boundary point for it.
Now define the family of stopping times
\begin{eqnarray}\nonumber
T_1 &:=& \inf \{ t>0 : X_t=0 \}, \\ \nonumber
T_n &:=& \inf \{ t>T_{n-1} +\varepsilon : X_t=0 \}, \quad n=2,3,\cdots,
\end{eqnarray}
for any fixed $\varepsilon>0$, to see that
\begin{equation}\label{solfam1}
X_t= \left( \sqrt{x_0} + \frac{W_t}{\sqrt{2}} \right)^2 \mathlarger{\mathlarger{\mathbbm{1}}}_{t < T_n}
\end{equation}
is a solution to~\eqref{stfbd} for any $n=1,2,\cdots$. Consider also
\begin{equation}\nonumber
dX_t=\sqrt{2 X_t} \circ dW_t, \qquad X_0=0,
\end{equation}
which admits as solution
\begin{equation}\label{solfam2}
X_t = \frac{(W_{t}-W_{\tau})^2}{2} \, \mathlarger{\mathlarger{\mathbbm{1}}}_{t > \tau}
\end{equation}
for any fixed $\tau \ge 0$. Combining these two results, \eqref{solfam1} and~\eqref{solfam2}, yields
an uncountable family of solutions
to~\eqref{stfbd}, unlike~\eqref{fbd2} that, by Proposition~\ref{exun}, possesses a unique solution.
Note also that the formal It\^o counterpart of~\eqref{stfbd}, i.e.
\begin{equation}\nonumber
dX_t= \frac12 \, dt + \sqrt{2 X_t} \, dW_t, \qquad X_0=x_0 > 0
\end{equation}
possesses, again by Proposition~\ref{exun}, a unique solution, which turns out to be given by~\eqref{solun}.
These facts illustrate how analytical issues play a role in the It\^o vs Stratonovich dilemma.

\section{Conclusions}
\label{conclusions}

We have illustrated our results, all throughout this work, with several examples that possess non-Lipschitz, but H\"older$-1/2$ continuous diffusion
terms. In particular, the square root function has played the key role in this respect in all of our examples. So the natural question that in this
moment arises is the following one: does this sort of diffusion term arise naturally in applications or could it be considered as a pathological
mathematical counterexample? To answer this question we can for instance follow van Kampen; according to him the diffusion term
``determines the magnitude of the fluctuations and must be found from physical considerations ($\cdots$).
For instance, the fluctuations in the number of electrons arriving on an anode will be roughly proportional to the square root of that number''~\cite{kampen}.
This observation therefore justifies the potential relevance of our results in physics. But actually the presence of square root diffusion terms is quite
standard in different applications. Let us briefly mention some examples of this fact.
The spatially extended version (i.e., the stochastic partial differential equation version) of the equation
\begin{equation}\nonumber
dX_t = \sqrt{X_t} \, dW_t,
\end{equation}
i.e. of the Feller branching diffusion, has been used to model the growth and motion of plankton populations~\cite{adler};
for related works (and in turn related to the critical Galton-Watson process) see~\cite{bahram,yrs}.
This equation has also been used in mathematical finance~\cite{ander}; for an extension of this model see~\cite{ander2}.
The equation
\begin{equation}\nonumber
dX_t = (\alpha -\beta X_t) \, dt + \sqrt{X_t(1-X_t)} \, dW_t,
\end{equation}
with $\beta \geq \alpha \geq 0$ has been considered in population genetics~\cite{feller2}; for the case $\alpha=\beta=0$ one can see~\cite{hof}.
The SDE
\begin{equation}\nonumber
dX_t = \gamma X_t(1-X_t) \, dt + \sqrt{X_t(1-X_t)} \, dW_t,
\end{equation}
with $\gamma > 0$, has been used to describe reaction processes, as well as its partial differential version
has been used to describe reaction-diffusion processes~\cite{doer}; related equations, both in the ordinary and partial
stochastic differential setting, can be found in~\cite{munoz}. It is also remarkable that some spatially extended versions
of this equation appear in the field of high energy physics~\cite{peschanski,prikhodko}.

During decades, the interpretation of noise dilemma has been regarded in the light of the transformation that takes equation~\eqref{its} into~\eqref{sfi} and its reverse counterpart~\cite{mmcc}. Herein we have seen that a class of SDEs, examples of which are relevant in physics and other applications, cannot be treated in these simple terms.
As a matter of fact, the presence of absorbing states is quite commonly accompanied by square root diffusion terms in these
physically motivated SDEs. The lack of Lipschitz continuity of the diffusion term precisely in these states implies that the
classical existence and uniqueness theorem cannot be applied and one has to rely instead on the Watanabe-Yamada theorem; but however
this theorem applies only to the It\^o interpretation of noise. Moreover, we have shown that these states are reached with probability one,
and even in finite mean time, in relevant examples, what in turn has strong consequences on the dynamics of a given SDE in the long
and even intermediate time. So, according to the results presented herein, for times potentially relevant in applications, the formal
use of the transformation in Theorem~\ref{itotost} or a change in the noise interpretation may turn a unique solution into an
uncountable number of them, or erase an absorbing state. The Wanatabe-Yamada theorem implies that the use of the It\^o interpretation
is safe in these cases, but however the Stratonovich interpretation may be affected by multiplicity of solutions as we have shown
in Section~\ref{feller}.
This is perhaps the reason why, in all the examples we have listed, the It\^o interpretation was chosen
(in those cases in which the interpretation of noise is explicitly mentioned), although our arguments,
to the best of our knowledge, are shown nowhere in those references (the arguments employed to choose this interpretation are
usually more related to modeling~\cite{mmcc}).
All in all, although the Stratonovich interpretation of noise has been sometimes preferred in the physical literature~\cite{kampen},
our results show that this is something that cannot be assumed in full generality.

In summary, we have shown that the It\^o vs Stratonovich dilemma cannot always be reduced to a redefinition of the drift term, and that
there are relevant physical examples in which more subtle stochastic analytical facts
(such as uniqueness/multiplicity of solution/s) have to be taken into account. As a matter of
fact, this drift redefinition is sometimes ill posed, and cannot be used to analyze a class of SDEs, even in the numerical sense.
As we have seen, the unjustified application of this result may lead to the disappearance of absorbing states, which meaning can be
as important as extinction in population dynamics or bankrupt in a financial system, in the SDE at hand.
Our results clarify as well why the It\^o interpretation has been chosen in a series of applied examples
from a stochastic analytical viewpoint.

\section*{Acknowledgements}

This work has been partially supported by the Government of Spain (Ministry of Economy, Industry, and Competitiveness) through Project MTM2015-72907-EXP.

\vskip5mm
\noindent
{\footnotesize
\'Alvaro Correales\par\noindent
Departamento de Matem\'aticas\par\noindent
Universidad Aut\'onoma de Madrid\par\noindent
E-28049, Madrid, Spain\par\noindent
{\tt alvaro.correales@estudiante.uam.es}\par\noindent
\& Department of Mathematics\par\noindent
Stockholm University\par\noindent
SE-10691, Stockholm, Sweden\par\noindent
{\tt alvaro@math.su.se}\par\vskip1mm\noindent
Carlos Escudero\par\noindent
Departamento de Matem\'aticas Fundamentales\par\noindent
Universidad Nacional de Educaci\'on a Distancia\par\noindent
E-28040, Madrid, Spain\par\noindent
{\tt cescudero@mat.uned.es}\par\vskip1mm\noindent
}

\begin{thebibliography} {99}

\bibitem{remark} The convergence was shown to be in law: clearly, the It\^o integral was not introduced previously to It\^o seminal work.

\bibitem{adler} R. Adler, {\it Superprocesess and Plankton Dynamics}, in: Monte Carlo Simulations in Oceanography,
Proceedings of the 9th `Aha Huliko`a Hawaiian Winter Workshop, University of Hawaii, 1997, 121--127.

\bibitem{ander} L. Andersen and J. Andreasen, {\it Volatility skews and extensions of the LIBOR market model},
Applied Mathematical Finance {\bf 7}, 1--32 (2000).

\bibitem{ander2} L. Andersen and V. V. Piterbarg, {\it Moment explosions in stochastic volatility models}, Finance and Stochastics {\bf 11}, 29--50 (2007).

\bibitem{munoz} F. Benitez, C. Duclut, H. Chat\'e, B. Delamotte, I. Dornic, and M. A. Mu\~noz,
{\it Langevin equations for reaction-diffusion processes}, Physical Review Letters {\bf 117}, 100601 (2016).

\bibitem{boro} A. N. Borodin and P. Salminen, {\it Handbook of Brownian Motion -- Facts and Formulae},
Birkh\"auser Verlag, Basel-Boston-Berlin, 2002.

\bibitem{doer} C. Doering, C. Mueller, and P. Smereka,
{\it Interacting particles, the stochastic Fisher-Kolmogorov-Petrovsky-Piscounov equation, and duality},
Physica A: Statistical Mechanics and its Applications {\bf 325}, 243--259 (2003).

\bibitem{feller1} W. Feller, {\it Die Grundlagen der Volterraschen Theorie des Kampfes ums Dasein in wahrscheinlichkeitstheoretischer Behandlung},
Acta Biotheoretica {\bf 5}, 11--40 (1939).

\bibitem{feller2} W. Feller, {\it Diffusion Processes in Genetics}, in:
Proceedings of the Second Berkeley Symposium on Mathematical Statistics and Probability, University of California Press, 1951, 227--246.

\bibitem{hell} I. Helland, {\it One-dimensional diffusion processes and their boundaries}, Statistical research report,
University of Oslo, 1996.

\bibitem{bahram} B. Houchmandzadeh, {\it Clustering of diffusing organisms}, Physical Review E {\bf 66}, 052902 (2002).

\bibitem{jean} M. Jeanblanc, M. Yor, and M. Chesney, {\it Mathematical Methods for Financial Markets}, Springer, Berlin, 2006.

\bibitem{kuo} H.-H. Kuo, {\it Introduction to Stochastic Integration}, Springer, New York, 2006.

\bibitem{mmcc} R. Mannella and P. V. E. McClintock,
{\it It\^o versus Stratonovich: 30 years later}. Fluctuation and Noise Letters {\bf 11}, 1240010 (2012).

\bibitem{nualart} D. Nualart, {\it The Malliavin Calculus and Related Topics}, Springer, Berlin, 1995.

\bibitem{oksendal} B. {\O}ksendal, {\it Stochastic Differential Equations: An Introduction with Applications}, Springer, Berlin, 2003.

\bibitem{peschanski} R. Peschanski, {\it Traveling wave solution of the Reggeon field theory}, Physical Review D {\bf 79}, 105014 (2009).

\bibitem{prikhodko} N. V. Prikhod'ko, {\it Homogeneous Balitsky-Kovchegov Hierarchy and Reggeon Field Theory}, in:
Proceedings of the XXI International Baldin Seminar on High Energy Physics Problems, 2012, 48.

\bibitem{revuz} D. Revuz and M. Yor, {\it Continuous Martingales and Brownian Motion}, Springer-Verlag, Berlin-Heidelberg, 1999.

\bibitem{stratonovich} R. L. Stratonovich, {\it A new representation for stochastic integrals and equations},
SIAM Journal on Control {\bf 4}, 362--371 (1966).

\bibitem{sv} D. W. Stroock and S. R. S. Varadhan,
{\it Multidimensional Diffusion Processes}, Springer-Verlag, Berlin-Heidelberg, 2006.

\bibitem{hof} T. D. Tran, J. Hofrichter, and J. Jost,
{\it An introduction to the mathematical structure of the Wright--Fisher model of population genetics}, Theory in Biosciences {\bf 132}, 73--82 (2013).

\bibitem{kampen} N. G. van Kampen, {\it It\^o versus Stratonovich}, Journal of Statistical Physics {\bf 24}, 175--187 (1981).

\bibitem{wy} T. Yamada and S. Watanabe, {\it On the uniqueness of solutions of stochastic differential equations},
Journal of Mathematics of Kyoto University {\bf 11}, 155--167 (1971).

\bibitem{yrs} W. R. Young, A. J. Roberts, and G. Stuhne, {\it Reproductive pair correlations and the clustering of organisms}, Nature {\bf 412}, 328--31 (2001).

\end{thebibliography}
\end{document}